\documentclass[12pt,notitlepage]{amsart}
\usepackage{latexsym,amsfonts,amssymb,amsmath,amsthm} 
\usepackage[inner=1.0in,outer=1.0in,bottom=1.0in, top=1.0in]{geometry}

\newcommand{\mz}{\ensuremath{\mathbb Z}}
\newcommand{\mr}{\ensuremath{\mathbb R}}

\newcommand{\mc}{\ensuremath{\mathbb C}}

\newcommand{\shortmod}{\ensuremath{\negthickspace \negthickspace \negthickspace \pmod}}

\newcommand{\half}{\ensuremath{ \frac{1}{2}}}

\newcommand{\intR}{\int_{-\infty}^{\infty}}

\newcommand{\thalf}{\tfrac12}
\newcommand{\sumstar}{\sideset{}{^*}\sum}
\newcommand{\leg}[2]{\left(\frac{#1}{#2}\right)}
\newcommand{\e}[2]{e\left(\frac{#1}{#2}\right)}

\newcommand{\Ai}{\mbox{Ai}}

\theoremstyle{plain}		
	\newtheorem{mytheo}{Theorem}[section]

	\newtheorem{mycoro}[mytheo]{Corollary}
     \newtheorem{mylemma}[mytheo]{Lemma}

\theoremstyle{remark}

\numberwithin{equation}{section}
\begin{document}
\title{The second moment of $GL(3) \times GL(2)$ $L$-functions, integrated}
\author{Matthew P. Young}
\address{Department of Mathematics \\
	  Texas A\&M University \\
	  College Station \\
	  TX 77843-3368 \\
		U.S.A.}
\email{myoung@math.tamu.edu}
\thanks{Work supported by NSF grant DMS-0758235}

\begin{abstract}
We consider the family of Rankin-Selberg convolution $L$-functions of a fixed $SL(3, \mz)$ Maass form with the family of Hecke-Maass cusp forms on $SL(2, \mz)$.  We estimate the second moment of this family of $L$-functions with a ``long'' integration in $t$-aspect.  These $L$-functions are distinguished by their high degree ($12$) and large conductors (of size $T^{12}$).
\end{abstract}

\maketitle

\section{Introduction}
In this paper we study the second moment of the Rankin-Selberg $L$-functions $L(\phi \times u_j, \half + it)$ of a fixed Hecke-Maass form $\phi$ on $SL(3, \mz)$ convolved with the family $u_j$ of Hecke-Maass cusp forms on $SL(2,\mz)$ as well as with the twists by $n^{it}$.  This family is ``large'' as measured in a variety of ways: there are $T^3$ elements in the family, each having degree $12$ and conductor of size $T^{12}$.  For comparison, the classical large sieve can estimate the eighth moment of the family of classical Dirichlet $L$-functions of modulus $q \leq Q$ (having $Q^2$ elements of degree $8$ and conductors of size $Q^8$). The various $GL(2)$ large sieve type inequalities also generally allow for degree $8$ $L$-functions with similarly-sized conductors, so that one can make a case that this family is significantly larger than others appearing in the literature.  For use in applications, it is highly desirable to have control over large families of harmonics, as they produce stronger detectors of arithmetical functions; see the introduction of \cite{DukeI2}. 

Another way to motivate interest in this particular family 
is that the first moment (at the central point 
$s=1/2$) was recently used by X. Li \cite{Li} \cite{Li2} to show subconvexity for a 
self-dual $GL(3)$ $L$-function in $t$-aspect (amongst other things).  The self-duality is crucially used to impose nonnegativity of the central values.  
In order to use moments to study non-self-dual forms, as well as Rankin-Selberg convolutions at points other than $s=1/2$, it seems necessary to study the second moment.  
However, this approach has substantial new difficulties. In particular, the second moment of this family at the central point has prohibitively large conductors (of size $T^{12}$ compared to $T^2$ elements in the family, a sixth power).  However, one can enlarge the size of the family without substantially growing the size of the conductors by twisting by $n^{it}$ with $t$ almost as large as the spectral parameter.  This procedure then brings the problem into the presumably more feasible range where the conductor is the fourth power of the size of the family.  Even so, the conductors of the family are still very large so that estimating this moment requires a substantial amount of cancellation.  In fact, the main difficulty is showing simultaneous cancellation in the twists by the Hecke-Maass forms as well as by $n^{it}$.  Many of the methods in the literature used to show cancellation in the Maass form aspect are incompatible with the $t$-aspect integration.

\begin{mytheo}
\label{thm:mainthm}
 We have
\begin{equation}
\label{eq:secondmoment}
 \int_{-T^{1-\varepsilon}}^{T^{1-\varepsilon}} \sum_{\substack{T < t_j \leq 2T \\ u_j \text{ even}}} |L(u_j \times \phi, \tfrac12 + i t)|^2 dt \ll T^{3 + \varepsilon}.
\end{equation}
\end{mytheo}
{\bf Remarks.}
\begin{itemize}
 \item
The conductor of $|L(u_j \times \phi, \half + i t)|^2$ is $\asymp T^{12}$, so that the convexity bound is recovered above using the method of Heath-Brown (Lemma 3 of \cite{H-B}).  
\item An easy consequence of Theorem \ref{thm:mainthm} is that for ``almost all'' $t_j$ in the family,  $$\int_{-T^{1 - \varepsilon}}^{T^{1-\varepsilon}}  |L(u_j \times \phi, \tfrac12 + i t)|^2 dt \ll T^{1 + \varepsilon},$$
in the sense that for any fixed $\varepsilon > 0$, the number of $t_j$ for which this bound is not satisfied is $o(T^2)$.
\item The reason $t$ is slightly smaller than $T$ is to avoid the intricate dependence of the conductor on $t_j$, though for $t= t_j$ see the companion paper \cite{Y}.  With some extra work it is likely that one could extend the left hand side of \eqref{eq:secondmoment} to $|t| \leq T, t_j \leq T$, but we have not carried out the details.  
\item The method of proof can also handle the analogous twists in the weight aspect by classical holomorphic modular forms on the full modular group; see Section 8 of \cite{ILS}.
\item The combined $t$-integral and spectral sum is reminiscent of Sarnak's work on the fourth moment of Gr\"{o}ssencharacter $L$-functions \cite{Sgross}.
\item Diaconu, Garrett, and Goldfeld \cite{DGG} have generalized the method of Good \cite{Good} to capture quantities of the form \eqref{eq:secondmoment}, but with certain weight functions (depending on $t_j$) in the integral.  It is difficult to asymptotically evaluate these weight functions,
so it is unknown what this implies about \eqref{eq:secondmoment}.
\item A. Venkatesh posed this problem during a problem session at the October 2006 AIM workshop on the subconvexity problem.
\end{itemize}

Theorem \ref{thm:mainthm} potentially represents progress towards subconvexity for these $L$-functions.  If one could shrink the family by any significant amount and still obtain a Lindel\"{o}f-consistent upper bound then one would obtain a subconvex bound.  This would require another source of cancellation. (Perhaps from solving the shifted convolution problem for a fixed $GL(3)$ Maass form?)  An additional problem is that one has to abandon the use of the large sieve inequality (a crucial ingredient in our proof) yet reclaim its substantial savings effect.  

A natural way to attack this problem is with a hybrid large sieve inequality of the form
\begin{equation}
\label{eq:hybridspectral}
\int_{-U}^{U} \sum_{t_j \leq T}  \Big| \sum_{n \leq N} a_n \lambda_j(n) n^{it} \Big|^2 dt \ll \Delta(N, T, U) (NTU)^{\varepsilon} \sum_{n \leq N} |a_n|^2.
\end{equation}
A simple application of Iwaniec's spectral large sieve \cite{Iw2} shows $\Delta(N,T, U) \ll U (N + T^2)$, and one would like to replace this by $N + U T^2$.  However, this appears to be an extremely difficult problem, and in fact in this generality it essentially implies the Ramanujan-Petersson conjecture for Maass forms! (To see this, take $U$ to be a large fixed power of $N$ to pick out the diagonal only on the left hand side and choose $a_n$ to select $n=N$ only, showing $|\lambda_j(N)|^2 \ll T^2 (TN)^{\varepsilon} \ll_{j} N^{\varepsilon}$.)
One might even consider a simpler problem where $U$ has restricted size (say $U \leq T$) and $N$ is large with respect to $T$.  Even this seems to be a difficult and highly interesting problem (in my opinion).  The state of affairs here for $GL(2)$ harmonics is quite different than for $GL(1)$ (multiplicative characters), where we do have the essentially optimal result of Gallagher \cite{Gallagher}
\begin{equation}
\label{eq:Gallagher}
 \int_{-U}^{U} \sum_{q \leq Q} \frac{q}{\phi(q)} \sumstar_{\chi \shortmod{q}} \Big| \sum_{n \leq N} a_n \chi(n) n^{it} \Big|^2 dt \ll (N+ UQ^2)  \sum_{n \leq N} |a_n|^2.
\end{equation}

The difficulty of estimating \eqref{eq:hybridspectral} for general coefficients $a_n$ is a barrier in our problem, which requires good estimates when the $a_n$'s are specialized to be coefficients of a $GL(3)$ $L$-function.  However, we have more available tools for these specific choices of coefficients, and in particular the $GL(3)$ Voronoi summation formula plays a key role.  X. Li \cite{Li} \cite{Li2} showed how this summation formula can be very powerful in the study of this family, but an attempt to directly generalize her approach on the first moment meets extreme difficulties with the second moment (see the first few sentences of Section \ref{section:Diophantine} below).


In our companion paper \cite{Y}, we considered this family of $L$-functions at the special point $\half + it_j$.  There are some similarities between the two problems but each requires substantially different ideas.  In particular, the analog of one of the key ideas in \cite{Y} (namely, Poisson summation in the variable $a$ modulo $b$ in  \eqref{eq:6.3} below) is not used here as it ultimately seemed not to not be substantially helpful, despite some promising hints.  Another major difference between these works is that in \cite{Y} we were able to appeal to some large sieve inequalities due to \cite{Luo} (after \cite{DI}), which we improved further, while here we could not use anything of the form \eqref{eq:hybridspectral}.  Furthermore, the conductors of the family of $T^2$ elements in \cite{Y} have size $T^{6}$ (the cube of the number of elements in the family as opposed to the fourth power appearing here), which has the effect that the $GL(3)$ Voronoi summation formula is relatively less powerful in this article than in the companion.  Indeed, the dual sum after Voronoi summation had essentially no length! (Though it should be stressed that this feature only occurred due to our improvement on the relevant large sieve inequality.)  This paper is independent of \cite{Y} (at the cost of some repetition).

\subsection{Acknowledgements}
I would like to thank Brian Conrey, Adrian Diaconu, Xiaoqing Li, Peter Sarnak, and Akshay Venkatesh for interesting discussions, and the referee for a very careful reading which greatly improved the paper.

\section{Notation}
See \cite{Goldfeld} for the material and notation on $GL(3)$ Maass forms.  Suppose $\phi$ is a Maass form for $SL(3, \mz)$ of type $(\nu_1, \nu_2) \in \mc^{2}$ 
which is an eigenfunction of all the Hecke operators.  The Godement-Jacquet $L$-function associated to $\phi$ is
\begin{equation*}
 L(\phi, s) = \sum_{n=1}^{\infty} \frac{A(1,n)}{n^s} = \prod_{p} (1 - A(1,p)p^{-s} + A(p,1)p^{-2s} - p^{-3s})^{-1}.
\end{equation*}
Here $A(m,n)$ are the Fourier coefficients normalized as in \cite{Goldfeld}.  In particular, $A(1,1) =1$ and $|A(m,n)|^2$ are constant on average (see Remark 12.1.8 of \cite{Goldfeld}).  The dual Maass form $\widetilde{\phi}$ is of type $(\nu_2, \nu_1)$ and has $A(n,m) = \overline{A(m,n)}$ as its $(m,n)$-th Fourier coefficient, whence
 $L(\widetilde{\phi}, s) = \sum_{n=1}^{\infty} A(n,1) n^{-s}$. 
Letting
\begin{equation*}
 \Gamma_{\nu_1, \nu_2}(s) = \pi^{-3s/2} \Gamma\leg{s + 1 - 2 \nu_1 - \nu_2}{2} \Gamma\leg{s + \nu_1 - \nu_2}{2} \Gamma\leg{s - 1 + \nu_1 + 2\nu_2}{2},
\end{equation*}
the functional equation for $L(\phi, s)$ reads
\begin{equation}
 \Gamma_{\nu_1, \nu_2}(s) L(\phi, s) = \Gamma_{\nu_2, \nu_1}(1-s) L(\widetilde{\phi}, 1-s).
\end{equation}

Let $(u_j)$ be an orthonormal basis of Hecke-Maass cusp forms on $SL(2,\mz)$ with corresponding Laplace eigenvalues $\frac14 + t_j^2$.  Let $\lambda_j(n)$ be the Hecke eigenvalue of the $n$-th Hecke operator for the form $u_j$.  Since the Hecke operators on $GL(2)$ are self-adjoint, the $\lambda_j(n)$'s are real.  Then
 $L(u_j,s) = \sum_{n=1}^{\infty} \lambda_j(n) n^{-s}$
satisfies a functional equation relating to $L(u_j,1-s)$.  To be clear, the Hecke operators are normalized so that the Ramanujan-Petersson bound is $|\lambda_j(p)| \leq 2$.  Say that the $n$-th Fourier coefficient of $u_j$ is $\rho_j(n)$, so that $\rho_j(n) = \rho_j(1) \lambda_j(n)$.  With the definition
\begin{equation}
\label{eq:alphaj}
 \alpha_j = \frac{|\rho_j(1)|^2}{\cosh(\pi t_j)},
\end{equation}
then $t_j^{-\varepsilon} \ll \alpha_j \ll t_j^{\varepsilon}$ due to \cite{HL} and \cite{I}.

As explained in Chapter 12.2 of \cite{Goldfeld}, the Rankin-Selberg convolution of $\phi$ and $u_j$ is
\begin{equation}
\label{eq:Rankin}
 L(u_j \times \phi, s) = \sum_{m,n =1}^{\infty} \frac{\lambda_j(n) A(m,n)}{(m^2 n)^s}.
\end{equation}
The completed $L$-function associated to $L(u_j \times \phi, s)$, for $u_j$ even, 
takes the form 
\begin{multline}
 \Lambda(u_j \times \phi, s) = \pi^{-3s} 
\Gamma\leg{s - i t_j - \alpha}{2} \Gamma\leg{s - i t_j - \beta}{2} \Gamma\leg{s - i t_j - \gamma}{2}
\\
\Gamma\leg{s + i t_j - \alpha}{2} \Gamma\leg{s + i t_j - \beta}{2} \Gamma\leg{s + i t_j - \gamma}{2}
L(u_j \times \phi, s),
\end{multline}
where $\alpha = -\nu_1 - 2 \nu_2 + 1$, $\beta = -\nu_1 + \nu_2$, and $\gamma = 2 \nu_1 + \nu_2 -1$ (see Theorem 12.3.6 of \cite{Goldfeld} for the explicit gamma factors).  Then this Rankin-Selberg convolution has a holomorphic continuation to $s \in \mc$ and satisfies the functional equation
\begin{equation}
 \Lambda(u_j \times \phi, s) = \Lambda(u_j \times \widetilde{\phi}, 1-s).
\end{equation}

\section{Basic tools}
\subsection{Approximate functional equation}
We shall use an approximate functional equation to represent the values of $L$-functions.  Write $\lambda_{u_j \times \phi}(n)$ for the coefficient of $n^{-s}$ in the Dirichlet series \eqref{eq:Rankin} for $L(u_j \times \phi, s)$.  Then Theorem 5.3 of \cite{IK} says for any $X > 0$,
\begin{equation}
\label{eq:AFE}
 L(u_j \times \phi, \tfrac12 + i t) = \sum_n \frac{\lambda_{u_j \times \phi}(n)}{n^{\half + i t}} V_{\half + i t}(n/X) + \epsilon_{t, t_j} \sum_n \frac{\lambda_{u_j \times \widetilde{\phi}}(n)}{n^{\half - i t}} V_{\half - i t}^*(nX),
\end{equation}
where $V_{s}(y)$ and $V_{s}^{*}(y)$ are 
certain explicit smooth functions, and $\epsilon_{t, t_j}$ is a certain complex number of absolute value $1$.  
More precisely,
\begin{equation}
\label{eq:V}
 V_{\half + it}(y) = \frac{1}{2 \pi i} \int_{(3)} y^{-s} \frac{\gamma(\half + i t + s)}{\gamma(\half + i t)} \frac{G(s)}{s} ds,
\end{equation}
where $\Lambda(u_j \times \phi, s) = \gamma(s) L(u_j \times \phi, s)$ and $G(s)$ is an entire function with rapid decay in the imaginary direction.  Here $V^{*}_{\half - i t}$ has a similar form to $V_{\half + it}$ but with $\gamma(s)$ replaced by $\gamma^*(s)$, where $\Lambda(u_j \times \widetilde{\phi}, s) = \gamma^*(s) L(u_j \times \widetilde{\phi}, s)$.

\subsection{The large sieve}
The classical large sieve inequality for Farey fractions states
\begin{equation}
\label{eq:largesieve}
 \sum_{b \leq B} \sum_{\substack{x \shortmod{b} \\ (x,b) = 1}} \Big| \sum_{N \leq m < N + M} a_m \e{xm}{b} \Big|^2 \leq (B^2 + M) \sum_{N \leq m < N+ M} |a_m|^2.
\end{equation}
For our purposes we require an additional oscillatory integral in the spirit of \cite{Gallagher}, but we could not find the following result in the literature.
\begin{mylemma}
\label{lemma:variantsieve}
Let $f(y)$ be a continuously differentiable function on $[N, N+M]$ such that $f'$ does not vanish.  Let $X = \sup_{y \in [N, N+M]} \frac{1}{|f'(y)|}$.  Then for any complex numbers $b_m$,
\begin{equation}
\label{eq:9.2}
 \int_{-T}^{T} \sum_{b \leq B} \sum_{\substack{x \shortmod{b} \\ (x,b) = 1}} \Big| \sum_{N \leq m < N + M} b_m \e{xm}{b} e(t f(m)) \Big|^2 dt \ll (B^2 T + X) \sum_{N \leq m < N+ M} |b_m|^2.
\end{equation}
\end{mylemma}
\noindent We reproduce the proof appearing in \cite{Y}.
\begin{proof}
By the change of variables $t \rightarrow tT$, it suffices to consider the case $T=1$.  Let $g$ be a nonnegative Schwartz function such that $g(x) \geq 1$ for $|x| \leq 1$, such that $\widehat{g}$ has compact support.  See \cite{Vaaler} for a nice survey on such functions as well as some ideas relevant in this proof. Then for any sequence of complex numbers $c_m$, we have 
\begin{equation}
\label{eq:9.3}
\int_{-1}^{1} \big| \sum_m c_m e(t f(m)) \big|^2 dt \leq \sum_{m,n} c_m \overline{c_n} \widehat{g}(f(m) - f(n)).
\end{equation}
Since $\widehat{g}$ is compactly supported, we must have $|f(m) - f(n)| \ll 1$.  By the mean-value theorem, $|f(m) - f(n)| \geq |m-n| \inf_y |f'(y)|$, so $|m-n| \ll X$.  Dissect the sum over $m$ and $n$ into boxes $I \times J$ of sidelength $\ll \min(M, X)$ so that the only relevant boxes $I \times J$ have $I$ and $J$ either equal or adjacent (``nearby'', say).  Thus the right hand side of \eqref{eq:9.3} equals
\begin{equation}
\label{eq:9.4}
 \sum_{I, J \text{ nearby}} \sum_{(m,n) \in I \times J} c_m \overline{c_n} \widehat{g}(f(m) - f(n)).
\end{equation}
Having enforced the condition that $I$ and $J$ are nearby, we then reverse the Fourier transform to express it in terms of $g$, getting that \eqref{eq:9.4} equals
\begin{equation}
\label{eq:9.5}
\intR g(t) \sum_{I, J \text{ nearby}} \sum_{(m,n) \in I \times J} c_m e(tf(m)) \overline{c_n e(tf(n))} dt.
\end{equation}
By Cauchy's inequality, \eqref{eq:9.5} is
\begin{equation*}
 \ll \intR g(t) \sum_{I} \Big| \sum_{m \in I} c_m e(t f(m)) \Big|^2 dt.
\end{equation*}
Specializing this to $c_m = b_m \e{xm}{b}$ and summing over $x$ and $b$ appropriately gives that the left hand side of \eqref{eq:9.2} (with $T=1$) is
\begin{equation}
\label{eq:9.7}
 \ll \intR g(t) \sum_{I} \sum_{b \leq B} \sum_{\substack{x \shortmod{b} \\ (x,b) = 1}} \Big| \sum_{m \in I} b_m \e{xm}{b} e(t f(m)) \Big|^2 dt.
\end{equation}
By \eqref{eq:largesieve} with $a_m = b_m e(t f(m))$, we complete the proof noting that \eqref{eq:9.7} is
\begin{equation*}
\ll \intR g(t) \sum_{I} (B^2 + \min(M, X)) \sum_{m \in I} |b_m|^2 \ll (B^2 + X) \sum_{N \leq m < N + M} |b_m|^2. \qedhere
\end{equation*}
\end{proof}

\subsection{Kuznetsov formula}
Our tool for summing over the spectrum is the following.
\begin{mylemma}[Kuznetsov] 
\label{lemma:Kuznetsov}
Suppose that $h$ is holmorphic in the region $|\text{Im}(r)| \leq \half + \delta$ and satisfies $h(r) = h(-r)$ and $|h(r)| \ll (1+ |r|)^{-2 -\delta}$ for some $\delta > 0$.  Then
 \begin{multline*}
  \sum_{j} \alpha_j \lambda_j(m) \lambda_j(n) h(t_j) + \frac{1}{\pi} \intR \frac{\sigma_{2ir}(m) \sigma_{2ir}(n)}{(mn)^{ir} |\zeta(1+2ir)|^2} h(r) dr 
\\
= \pi^{-2} \delta_{m=n} \intR r \tanh(\pi r) h(r) dr + \sum_{c=1}^{\infty} \frac{S(m,n;c)}{c} \check{h}\leg{4 \pi \sqrt{mn}}{c},
 \end{multline*}
where $\alpha_j$ is defined by \eqref{eq:alphaj} and 
\begin{equation*}
 \check{h}(x) = \frac{2i}{\pi} \intR \frac{r h(r)}{\cosh(\pi r)} J_{2ir}(x) dr. 
\end{equation*}
\end{mylemma}

\section{Initial cleaning}
Throughout this article it is very convenient to refer to functions $f$ satisfying the following bounds
\begin{equation}
\label{eq:fbound}
x^k f^{(k)}(x) \ll_{k,C} (1 + \frac{|x|}{Y})^{-C},
\end{equation}
for each $k, C \geq 0$, and some parameter $Y$.  

The following technical lemma gives a pointwise upper bound on an $L$-function with a weight function that only loosely depends on its parameters.
\begin{mylemma}
\label{lemma:cleaning}
 Let $\varepsilon > 0$, suppose that $T^{\varepsilon} \leq U \leq \Delta \leq T^{1-\varepsilon}$, and $|t| \leq U$, $T < t_j \leq T + \Delta$.
Then there exist finitely many functions $W_k$ independent of $t$ and $t_j$ (but depending on $U$ and $T$) satisfying \eqref{eq:fbound} with $Y = 1$,
such that for some fixed interval $[X_0, X_1] \subset (0, \infty)$, we have
\begin{equation}
\label{eq:LAFEUB}
 |L(u_j \times \phi, \thalf + it)|^2 \ll \sum_{k} \int_{X_0}^{X_1} \Big|\sum_n \frac{\lambda_{u_j \times \phi }(n)}{n^{\half + it}} W_k(n/(XT^3)) \Big|^2 dX + O(T^{-200}).
\end{equation}
\end{mylemma}
\begin{mycoro}
\label{coro:cleaning}
Let $\varepsilon > 0$, suppose that $T^{\varepsilon} \leq U \leq \Delta \leq T^{1-\varepsilon}$, and $|t| \leq U$, $T < t_j \leq T + \Delta$.  Let
\begin{equation}
 \mathcal{M}(T, U, \Delta) = \int_{-U}^{U} \sum_{\substack{T < t_j \leq T + \Delta \\ u_j \text{ even}}} |L(u_j \times \phi, \tfrac12 + i t)|^2 dt.
\end{equation}
Then there exists $W$ independent of $t$ and $t_j$ (but depending on $U$ and $T$) satisfying \eqref{eq:fbound} with $Y=1$ such that
\begin{equation}
\label{eq:MTUUB}
 \mathcal{M}(T, U, \Delta) \ll \int_{-U}^{U} \sum_{\substack{T < t_j \leq T + \Delta \\ u_j \text{ even}}} |\sum_n \frac{\lambda_{u_j \times \phi }(n)}{n^{\half + it}} W(n/T^3)|^2 dt + O(T^{-100}).
\end{equation}
\end{mycoro}
To prove this Corollary, take the $k$-sum and $X$-integral outside the $t_j$-sum and $t$-integral, bound it by the supremum (if we do not do this the choices of $X$ and $k$ depend on $t, t_j$), and redefine $W$.

\begin{proof}[Proof of Lemma \ref{lemma:cleaning}]
We use a method similar to that of Section 5 of \cite{Y}, which we modified from \cite{DI}.

 Begin by writing the approximate functional equation \eqref{eq:AFE} in the shorthand form $L(u_j \times \phi, \thalf + it) = \sum_n a_n V(n/X) + \epsilon \sum_n b_n V^*(nX)$.  This is valid for all $X > 0$, so we freely integrate against $X^{-1} dX$ from $X = e^{-\half}$ to $X=e^{\half}$.  Then by Cauchy's inequality we deduce
\begin{equation}
\label{eq:LUB}
 |L(u_j \times \phi, \thalf + it)|^2 \leq 2 \int_{e^{-\half}}^{e^{\half}} |\sum_n a_n V(n/X)|^2 \frac{dX}{X} + 2 \int_{e^{-\half}}^{e^{\half}} |\sum_n b_n V^*(nX)|^2 \frac{dX}{X}.
\end{equation}
Notice that changing variables $X \rightarrow X^{-1}$ in the latter term makes the two terms more symmetric.  Indeed, the latter term then becomes the same as the former term with $\phi$ replaced by its dual and $t$ replaced by $-t$.  We focus on the former term.

We find a simpler expression for $V(x)$.
By Stirling's approximation, we have for $\lambda \in \mc$ fixed, $\text{Re}(s) > 0$, $s$ small compared to $T$, and with the shorthand $Z = \half + it \pm i t_j - \lambda$, that for certain polynomials $P_n$, we have
\begin{equation*}
\log \frac{\Gamma\left(\frac{Z+s}{2} \right)}{\Gamma\left(\frac{Z}{2} \right)}
= \frac{s}{2} \log(\frac{Z}{2}) + \frac{P_1(s)}{Z} + \frac{P_2(s)}{Z^2} + \dots + O\big(\frac{P_k(s)}{Z^k}\big),
\end{equation*}
With $Z' = \half + it \mp t_j - \lambda$, we have
\begin{equation*}
\half \log(\frac{Z}{2}) + \half \log(\frac{Z'}{2}) = \log |Z/2| + \half \log(1 + \frac{Z'-\overline{Z}}{\overline{Z}}).
\end{equation*}
A computation shows that
\begin{equation*}
 \half \log(1 + \frac{Z'-\overline{Z}}{\overline{Z}}) = \frac{Q_1(t)}{t_j} + \frac{Q_2(t)}{t_j^2} + \dots + O\big(\frac{Q_k(t)}{t_j^k}\big),
\end{equation*}
where each $Q_l$ is a polynomial of degree $l$.  Combining these estimates for the gamma factors, we obtain the asymptotic expansion
\begin{equation}
\label{eq:Vasymp}
 V(x) = \frac{1}{2 \pi i} \int_{(\sigma)} (\pi^3 x)^{-s} \frac{G(s)}{s} |q(t,t_j)|^{s/2} \Big(1 + \frac{R_1(s,t)}{t_j} + \dots + O\big(\frac{R_k(s,t)}{t_j^k} \big)\Big) ds,
\end{equation}
where each $R_l(s,t)$ is a polynomial in $s$ and $t$ of degree at most $l$ in terms of $t$, and $q(t,t_j)$ is the product of three $Z$ and three $Z'$ terms with $\lambda$ replaced by $\alpha, \beta, \gamma$.  Note $|q(t,t_j)| \asymp T^6$, uniformly in $t$ and $t_j$.
To be clear, we should choose a $G(s)$ that decays rapidly for $\text{Im}(s)$ large, such as $e^{s^2}$, truncate the $s$-integral at say $\log^2{T}$, apply Stirling's formula, and then relax the truncation, all with an acceptable error.

The representation \eqref{eq:Vasymp} with $\sigma > 0$ very large shows that we may truncate the sum over $n$ at $T^{3 + \varepsilon}$ with an acceptable error term.  With this truncation in place, we then fix $\sigma = 1$ and insert the asymptotic expansion \eqref{eq:Vasymp} into \eqref{eq:LUB}.  According to \eqref{eq:Vasymp}, write $V = V_0+ O(T^{3 \sigma} \leg{U}{T}^k)$ (so that $V_0$ accounts for all the terms in the expansion except for the error term).  
For $k$ sufficiently large in terms of $U$ and $T$, this error term is acceptable.  Then change variables $X \rightarrow X \pi^3 T^3/|q(t,t_j)|^{1/2}$ and by positivity extend the $X$-integral to a fixed interval, say $[X_0, X_1] \subset (0, \infty)$.  Now
\begin{multline*}
 |L(u_j \times \phi, \thalf + it)|^2 \ll \int_{X_0}^{X_1} \Big|\sum_n \frac{\lambda_{\phi \times u_j}(n)}{n^{\half + it}} \int_{(1)} \frac{G(s)}{s} \leg{X T^3}{n}^s (1 + 
\dots + \frac{R_{k-1}(s,t)}{t_j^{k-1}}) ds  \Big|^2 dX
\\
+ (\text{similar term}) 
+ O(T^{-200}).
\end{multline*}
Here the sum over $n$ is truncated at $T^{3 + \varepsilon}$, but since the contour can be shifted far to the right, we may relax this condition without making a new error term.  Next write $R_l(s,t) = R_{0,l}(s) + R_{1,l}(s) t + \dots + R_{l,l}(s) t^l$, and use Cauchy's inequality on the asymptotic expansion to obtain
\begin{equation*}
 |L(u_j \times \phi, \thalf + it)|^2 \ll \int_{X_0}^{X_1} \sum_{i \leq l < k} \frac{U^{2i}}{T^{2l}} \Big|\sum_n \frac{\lambda_{\phi \times u_j}(n)}{n^{\half + it}} V_{i,l}(n/(XT^3))\Big|^2 dX + (\text{similar}) + O(T^{-200}),
\end{equation*}
where
\begin{equation*}
 V_{i,l}(x) = \frac{1}{2 \pi i} \int_{(1)} \frac{G(s)}{s} R_{i,l}(s) x^{-s} ds.
\end{equation*}
Clearly each $V_{i,l}$ satisfies \eqref{eq:fbound} with $Y=1$.  The ``similar'' term has the same form as the displayed term, except the coefficients are conjugated and the weight function is different (the Langlands parameters $\alpha, \beta, \gamma$ are changed) yet it satisfies the same bounds (the proof was for any $\alpha, \beta, \gamma$ fixed), so again it takes the same form as the right hand side of \eqref{eq:LAFEUB}.  The total number of terms in the asymptotic expansion is finite.
\end{proof}

\begin{mytheo}
\label{thm:mainthm'}
 With conditions as in Lemma \ref{lemma:cleaning}, we have
\begin{equation}
 \mathcal{M}(T, U, \Delta) \ll T^{3 + \varepsilon}.
\end{equation}
\end{mytheo}
From Theorem \ref{thm:mainthm'}, we deduce Theorem \ref{thm:mainthm} by taking $U = T^{1-\varepsilon}$, and covering the interval $[T, 2T]$ with $O(T^{\varepsilon})$ subintervals of the form $[T', T' + \Delta]$ with $\Delta = T^{1-\varepsilon}$.

\begin{mylemma}
\label{lemma:cleaning2}
Let
\begin{equation}
\label{eq:h}
 h(r) = \frac{r^2 + \frac14}{T^2} \Big[\exp\big(-\big(\frac{r-T}{\Delta}\big)^2 \big) + \exp\big(-\big(\frac{r+T}{\Delta}\big)^2 \big)\Big].
\end{equation}
Then there exists a smooth function $w$ with support in a dyadic interval $[P, 2P]$ with 
\begin{equation}
P \ll T^{3 + \varepsilon}
\end{equation}
satisfying \eqref{eq:fbound} with $Y = P$
such that
\begin{equation*}
 \mathcal{M}(T, U, \Delta) \ll T^{\varepsilon} H(U, T, \Delta, w) + O(T^{-100}),
\end{equation*}
where
\begin{equation*}
 H(U, T, \Delta, w) = \int_{-U}^{U} \sum_{t_j} \alpha_j h(t_j) \Big| \sum_n \frac{\lambda_{u_j \times \phi}(n)}{n^{\half + it}} w(n) \Big|^2 dt.
\end{equation*}
\end{mylemma}
\begin{proof}
We start with \eqref{eq:MTUUB}.
From the lower bound $\alpha_j \gg t_j^{-\varepsilon}$ we may attach the weight $\alpha_j$ at the cost of $O(T^{\varepsilon})$.  
By positivity, we freely attach the weight function $h$ and extend the summation to all $t_j$.  Finally, we apply a smooth dyadic partition of unity to the inner $n$ sum, and apply Cauchy's inequality to this sum over the partition (only the terms with $P \ll T^{3 + \varepsilon}$ are relevant by a trivial bound),
completing the proof.
\end{proof}

\begin{mylemma}
\label{lemma:HUB2}
Let $g$ be a fixed nonnegative, even, Schwartz function satisfying $g(t) \geq 1$ for $|t| \leq 2$ whose Fourier transform is compactly supported, and 
 define for any finite sequence of complex numbers $b_n$,
\begin{equation}
\label{eq:5.2}
 \mathcal{H}(U, T, \Delta ; b_n) = \int_{-\infty}^{\infty} g(t/U) \sum_{t_j} \alpha_j h(t_j) \Big| \sum_n \lambda_{j}(n) n^{-it} b_n \Big|^2 dt.
\end{equation}
Then for some $L$ with
\begin{equation}
 L \ll \sqrt{P},
\end{equation}
and with coefficients 
\begin{equation}
\label{eq:andef}
 a_{n,l} = n^{-\half} A(l,n) w_L(n),
\end{equation}
where $w_L(n)$ has support in $n \asymp P/L^2$ and satisfies \eqref{eq:fbound} with $Y= P/L^2$, we have
\begin{equation}
\label{eq:HUB}
 H(U, T, \Delta, w) \ll T^{\varepsilon} \sum_{l \asymp L} l^{-1} \mathcal{H}(U, T, \Delta ; a_{n,l}) + O(T^{-200}).
\end{equation}
\end{mylemma}
\begin{proof}
We begin by writing $\lambda_{u_j \times \phi}(n)$ in terms of $\lambda_j(n)$ and $A(l,n)$ using \eqref{eq:Rankin}, and using Cauchy's inequality on the sum over $l$.
Then we break up the sum over $l$ into $O(\log{T})$ dyadic segments with $l \asymp L \ll \sqrt{P}$.

Considering the value of $L$ which maximizes the bound, this gives a bound of the form \eqref{eq:HUB} except that the weight function is of the form $w(l^2 n)$ which unfortunately depends on $l$.  We remove the $l$-dependence by multiplying $w(l^2 n)$ by $w_1(n)$, say, which satisfies \eqref{eq:fbound} with $Y = N/L^2$, is supported in an interval of the form $n \asymp Y$, and is identically one on the union of the supports of $w(l^2 n)$, for $l \asymp L$.  Then we separate variables with the Mellin technique, writing
\begin{equation*}
w_1(n) w(l^2 n)  = w_1(n) \frac{1}{2 \pi} \intR \widetilde{w}(iy) (l^2 n)^{-iy} dy.
\end{equation*}
By the rapid decay of $\widetilde{w}$, we may truncate the integral at $|y| \leq T^{\varepsilon} \leq U$ with an acceptable error.  Then we apply Cauchy-Schwarz to take the $y$-integral to the outside and change variables $t \rightarrow t-y$.  By positivity, we extend the $t$-integral to $|t| \leq 2U$, and integrate trivially over $y$.  Finally, by positivity we attach the weight function $g(t/U)$ and extend the integral to $\mr$.
\end{proof}
Set
\begin{equation}
\label{eq:N}
N = P/L^2,
\end{equation}
so that $w_L$ satisfies \eqref{eq:fbound} with $Y=N$.

With the above reductions, our goal for the rest of this paper is to prove the following.
\begin{mytheo}
\label{thm:mainthm''}
 With $a_{n,l}$ given by \eqref{eq:andef}, we have
\begin{equation}
\label{eq:4.4}
 \sum_{l \asymp L} l^{-1} \mathcal{H}(U, T, \Delta ; a_{n,l}) \ll T^{3 + \varepsilon} \sum_{l^2 n \ll T^{3 + \varepsilon}} \frac{|A(l,n)|^2}{ln}.
\end{equation}
\end{mytheo}
We now briefly explain how Theorem \ref{thm:mainthm''} implies Theorem \ref{thm:mainthm'} (and hence Theorem \ref{thm:mainthm}).

Using the polynomial growth of the Rankin-Selberg convolution $L(\phi \times \phi, s)$, one can show that (see Remark 12.1.8 of \cite{Goldfeld})
\begin{equation*}
 \sum_{l^2 n \leq x} |A(l,n)|^2 \ll x.
\end{equation*}
Actually, if $\phi$ is a Hecke eigenform then one can include more terms and obtain
\begin{equation}
\label{eq:lnA}
 \sum_{ln \leq x} |A(l,n)|^2 \ll x^{1 + \varepsilon},
\end{equation}
which we explain now.  The Hecke relations for $\phi$ imply
\begin{equation*}
 A(l,n) = \sum_{d | (l,n)} \mu(d) A(l/d, 1) A(1,n/d).
\end{equation*}
Inserting this into the left hand side of \eqref{eq:lnA}, applying Cauchy's inequality to the sum over $d$, and using the standard divisor function bound, we quickly obtain \eqref{eq:lnA}.

One of the basic techniques used throughout this paper is to apply an asymptotic expansion to a particular quantity and reduce the estimation of the entire quantity to that of the leading-order term, as the lower-order terms have all the essential characteristics of the main term yet are of smaller magnitude.

For the rest of the paper we fix
\begin{equation}
\label{eq:Delta}
\Delta = T^{1-\varepsilon}.
\end{equation}

\section{Applying the Kuznetsov formula}
\begin{mylemma}
\label{lemma:K0}
 Let $K_0(U, T, \Delta; b_{n,l})$ denote a sum of the form
\begin{multline}
\label{eq:K0}
 K_{0} = \frac{\Delta T}{\sqrt{N}} \intR g(t/U) \sum_{m,n} b_{m,l} \overline{b_{n,l}} \leg{m}{n}^{-it} 
\\
\sum_{c\leq \frac{N T^{\varepsilon}}{\Delta T}} \frac{S(m,n;c)}{\sqrt{c}} \e{- 2 \sqrt{mn}}{c} e^{\frac{ i T^2 c}{2 \pi \sqrt{mn}}} w_1\leg{c \Delta T}{\sqrt{mn}},
\end{multline}
where $g$ is as in Lemma \ref{lemma:HUB2},
\begin{equation}
\label{eq:bnl}
 b_{n,l} = \frac{A(l,n)}{\sqrt{n}} w(n),
\end{equation}
where $w$ has support in the dyadic interval $[N, 2N]$ and satisfies \eqref{eq:fbound} with $Y = N$, and $w_1$ is a smooth function on $\mr^{+}$ satisfying \eqref{eq:fbound} with $Y=1$.
If for all $L \ll T^{3/2 + \varepsilon}$ we have
\begin{equation}
\label{eq:K0bound}
\sum_{l \asymp L} l^{-1} | K_0(U, T, \Delta; b_{n,l})| \ll T^{3 + \varepsilon} \sum_{l^2 n \ll T^{3 + \varepsilon}} \frac{|A(l,n)|^2}{ln},
\end{equation}
then Theorem \ref{thm:mainthm''} holds.
\end{mylemma}
Our strategy of proof for Theorem \ref{thm:mainthm''} is thus to show \eqref{eq:K0bound} holds.

\begin{proof}
Our first step in estimating $\mathcal{H}$ is to apply Lemma \ref{lemma:Kuznetsov}.
It is a somewhat involved task to analyze the integral transform $\check{h}$, but Jutila and Motohashi \cite{JM} have obtained a precise asymptotic expansion of $\check{h}$ for the particular choice \eqref{eq:h}.
By (3.19) of \cite{JM}, we obtain an asymptotic expansion for $\check{h}(x)$ with leading term
\begin{equation}
\label{eq:hcheck}
\frac{4}{\pi} \sqrt{\frac{2}{x}} \Delta T  \exp\left(-\leg{2 \Delta T}{x}^2 \right) \cos\left(x - 2 T^2 x^{-1} +\frac{\pi}{4} \right).
\end{equation}
Actually they wrote the expansion in terms of a critical point $u_0$ solving $\sinh u_0 = \frac{2T}{x}$, which has the expansion $u_0 = \frac{2T}{x} + O(\leg{T}{x}^3)$; in fact, $u_0$ is holomorphic in terms of $T/x$.  Strictly speaking, the asymptotic is of the form $\cos(x - 2 T^2 x^{-1} + \frac{\pi}{4} + O(T^4 x^{-3}))$, which can be expanded into power series with leading term \eqref{eq:hcheck} provided $\Delta \geq T^{1/3 + \varepsilon}$.


Applying the Kuznetsov formula gives that
\begin{equation}
\label{eq:Kl}
 \mathcal{H}(U, T, \Delta;a_{n,l})  + (\text{Eisenstein}) = D + K(U, T, \Delta; a_{n,l}),
\end{equation}
say, where $D$ corresponds to the diagonal term and $K$ is the sum of Kloosterman sums.  The Eisenstein contribution is nonnegative and can be discarded for purposes of estimation of $\mathcal{H}$.  An easy computation gives
\begin{equation*}
 D \ll U \Delta T^{} \sum_{n \ll N} \frac{|A(l,n)|^2}{n},
\end{equation*}
which is sufficient for the goal of \eqref{eq:4.4}.

Inserting the asymptotic expansion for $\check{h}(x)$, noting that the exponential decay in \eqref{eq:hcheck} naturally allows the truncation
\begin{equation}
\label{eq:cbound}
 c \leq \frac{N T^{\varepsilon}}{\Delta T},
\end{equation}
and writing $2\cos(y) = e^{iy} + e^{-iy}$, we obtain an analogous asymptotic expansion for $K$ of the form
\begin{equation*}
 K = K_1 + K_{-1} + \dots + K_r  + K_{-r} + O(T^{-200}),
\end{equation*}
for some absolute constant $r$, 
where each $K_{i}$ is of the form \eqref{eq:K0}, and by a simple symmetry argument each $|K_{-i}| = |K_{i}|$.
In fact, the lower-order terms would have a weight function $f$ that is smaller by a certain power of $T$, but it only complicates the notation to include this behavior.
Thus Theorem \ref{thm:mainthm''} follows from \eqref{eq:K0bound}.
\end{proof}

\section{Diophantine approximation}
\label{section:Diophantine}
The extreme oscillation of the term $\e{-2 \sqrt{mn}}{c}$ is a source of difficulty in exploiting cancellation in the sum over $m$.  For instance, an application of the $GL(3)$ Voronoi formula would lead to a sum where the dual variable has size $\approx T^6$ (for ``typical'' choices of the parameters), which is a catastrophic loss (though there turns out to be a gain in the simplicity of the arithmetical properties of the new sum).  It seems necessary to somehow dampen the oscillations of this exponential.  To do so, note that the $t$-integral forces $m$ and $n$ to be close: essentially $m = n(1 + O(U^{-1}))$, so that the identity $-2 \sqrt{mn} = -m -n + (\sqrt{m}-\sqrt{n})^2$ implies a close approximation (note that $(\sqrt{m}-\sqrt{n})^2 \ll (m-n)^2/N \ll N U^{-2}$).  Although $\e{-m-n}{c}$ is just as oscillatory as $\e{-2\sqrt{mn}}{c}$ (meaning the arguments of the exponential are of the same order of magnitude), it has the property of being periodic in $m$ and $n$ modulo $c$ so that one can treat this term arithmetically and absorb the ``remainder'' $\e{(\sqrt{m} - \sqrt{n})^2}{c}$ into the weight function, which is much less oscillatory.  This is a key idea in this paper.  The papers \cite{CI} and \cite{Li} also found arithmetical features of this phase but only after applying a summation formula.

Another pleasant feature of $\e{-m-n}{c}$ is that $m$ and $n$ are naturally separated.  However, this ``twisting'' of the Kloosterman sum by the exponential $\e{-m-n}{c}$ has the side effect of creating terms of the form $\e{(h-1)m}{c}$ where $h$ is coprime to $c$, so that it does not always hold that $h-1$ is coprime to $c$, thus making the use of the large sieve inequalities or Voronoi summation problematic.  Naturally one can factor out the greatest common divisor of $h-1$ and $c$ to proceed further; that is the content of the following
\begin{mylemma}  For all integers $m$, $n$, and positive integers $c$, we have
 \begin{equation}
\label{eq:Stwist}
  S(m,n;c) \e{-m-n}{c} = \sum_{ab =c} \sum_{\substack{x \shortmod{b} \\ (x(x+a),b)=1}} \e{\overline{x} m - (\overline{x+a})n}{b}.
 \end{equation}
\end{mylemma}
Remark.  This calculation seems to have been first performed by Luo \cite{Luo}, but see also \cite{IL} for some curious connections.
\begin{proof}
 By opening the Kloosterman sum, the left hand side of \eqref{eq:Stwist} is
\begin{equation*}
 \sum_{\substack{h \shortmod{c} \\ (h,c)=1}} \e{(h-1)m + (\overline{h}-1)n}{c}.
\end{equation*}
Write $(h-1, c) = a$, $c=ab$, and change variables $h \equiv 1 + a x \pmod{c}$ where now $x$ runs modulo $b$ and satisfies $(x(1+ax),b) = 1$.  Note that $\overline{1+ax} - 1 \equiv -ax (\overline{1 + ax}) \pmod{b}$.  Replacing $x$ by $\overline{x}$ gives \eqref{eq:Stwist}.
\end{proof}

Inserting \eqref{eq:Stwist} into \eqref{eq:K0} gives
\begin{equation}
\label{eq:6.3}
 K_{0} = \frac{\Delta T}{\sqrt{N}}  \sum_{ab \leq \frac{N T^{\varepsilon}}{\Delta T}} \frac{1}{\sqrt{ab}} \sum_{\substack{x \shortmod{b} \\ (x(x+a),b)=1}} \sum_{m,n} b_m \overline{b_n} \e{\overline{x} m - (\overline{x+a})n}{b} Z(m,n),
\end{equation}
where
\begin{equation}
 Z(m,n) = e^{\frac{i T^2 ab}{2 \pi \sqrt{mn}}} w_1\leg{ab \Delta T}{\sqrt{mn}} \e{(\sqrt{m}-\sqrt{n})^2}{ab}  \intR g(t/U) \leg{m}{n}^{-it} dt.
\end{equation}
Let
\begin{equation}
\label{eq:J}
 J = \sum_{m,n} b_m \overline{b_n} \e{\overline{x} m - (\overline{x+a})n}{b} Z(m,n),
\end{equation}
so that
\begin{equation}
\label{eq:6.6}
 K_{0} \ll \frac{\Delta T}{\sqrt{N}}  \sum_{ab \ll \frac{N}{\Delta T}T^{\varepsilon} } \frac{1}{\sqrt{ab}} \sum_{\substack{x \shortmod{b} \\ (x(x+a),b)=1}} |J|.
\end{equation}

\section{Separation of variables}
\label{section:separation}
Next we want to separate the variables $m$ and $n$ in $Z(m,n)$, for which we shall use oscillatory integral transforms.  This will allow us to express $J$ in terms of a bilinear form so that we can apply the powerful technology of the large sieve inequality.  

\begin{mylemma}
\label{lemma:K00}
 Let $K_{00}(U, T, \Delta; b_{n,l}) $ denote an expression of the form
\begin{equation}
\label{eq:K00def}
 K_{00} = \Delta T^{1 + \varepsilon}  \sum_{ab \leq  \frac{NT^{\varepsilon}}{\Delta T}} \frac{1}{ab} \int\limits_{v \ll 1}  \sum_{\substack{x \shortmod{b} \\ (x,b)=1}} \Big| \sum_m b_m \e{\overline{x}m}{b} m^{iy_0} e\Big(\frac{v \sqrt{mN}}{Uab} \Big) \Big|^2 dv,
\end{equation}
where $b_m = b_{m,l}$ satisfies the condition \eqref{eq:bnl} as in Lemma \ref{lemma:K0},
and $y_0 \ll T^{\varepsilon}$ is fixed.  If 
\begin{equation}
 \label{eq:K00bound}
\sum_{l \asymp L} l^{-1} K_{00}(U, T, \Delta; b_{n,l}) \ll T^{3 + \varepsilon} \sum_{l^2 n \ll T^{3 + \varepsilon}}  \frac{|A(l,n)|^2}{ln},
\end{equation}
then Theorem \ref{thm:mainthm''} holds.
\end{mylemma}
\begin{proof}
 We find an asymptotic expansion $J = \sum_{1 \leq i \leq r} J_i + O(T^{-200})$ such that when inserted into \eqref{eq:6.6}, gives a bound of the form, say $|K_0| \ll \sum_i |K_{0, i}| + O(T^{-150})$, and such that each $|K_{0,i}|$ is of the form \eqref{eq:K00def}.  In this way we reduce the estimation of $K_0$ to that of $K_{00}$.

First attach a smooth, compactly-supported weight function $w_2(\sqrt{mn}/N)$ to $Z(m,n)$ that takes the value $1$ for all $m$ and $n$ in the support of the implicit weight function appearing in $b_m \overline{b_n}$, and write
$Z(m,n) = Z_1(m,n) Z_2(m,n)$, where
\begin{equation}
 Z_1(m,n) = e^{\frac{i T^2 ab}{2 \pi \sqrt{mn}}} w_1\leg{ab \Delta T}{\sqrt{mn}} w_2\leg{\sqrt{mn}}{N}, 
\end{equation}
and
\begin{equation}
 Z_2(m,n) = \e{(\sqrt{m}-\sqrt{n})^2}{ab} \intR g(t/U) \leg{m}{n}^{-it} dt.
\end{equation}

We separate variables in each of $Z_1$ and $Z_2$ in turn.
By Mellin inversion, we obtain
\begin{equation}
\label{eq:Z1}
 Z_1(m,n) = \frac{1}{2 \pi } \intR F_{a,b,\Delta, T, N}(iy) \leg{N}{\sqrt{mn}}^{iy} dy,
\end{equation}
where
\begin{equation*}
 F_{a,b,\Delta, T, N}(iv) = \int_0^{\infty}  w_1\leg{ab \Delta T}{N x} x^{-1} w_2(x) e^{\frac{i T^2 ab}{2 \pi N x}} x^{iv} dx. 
\end{equation*}
If $|y| \gg T^{\varepsilon}$ then repeated integration by parts shows $F(iy)$ is smaller than any negative power of $T$.  For $|y| \ll T^{\varepsilon}$, a trivial bound shows $F(iy) \ll 1$.  
This separates the variables $m$ and $n$ by a very short integral (essentially no cost).

Next we separate variables in
\begin{equation}
\label{eq:7.7}
 Z_2(m,n) =  \e{(\sqrt{m}-\sqrt{n})^2}{ab} U \widehat{g}\left(\tfrac{U}{2\pi} \log\frac{m}{n} \right).
\end{equation}
Recalling that $\widehat{g}$ has compact support, which restricts $m$ and $n$ so that $|m-n| \ll U^{-1} N$, we can then use a Taylor expansion to write
\begin{equation}
\label{eq:7.8}
 \log{\frac{m}{n}} = 2 \log\left( 1 + \frac{\sqrt{m} - \sqrt{n}}{\sqrt{n}} \right) = 2 \frac{\sqrt{m} - \sqrt{n}}{\sqrt{n}} -  \leg{\sqrt{m} - \sqrt{n}}{\sqrt{n}}^2 + \dots,
\end{equation}
noting
 $\left| \frac{\sqrt{m} - \sqrt{n}}{\sqrt{n}} \right| \ll U^{-1}$.
Then we get an asymptotic expansion for $\widehat{g}$ in the form
\begin{equation}
\label{eq:7.10}
 \widehat{g}\left(\frac{U}{2\pi} \log\frac{m}{n} \right) = \widehat{g}\left(\frac{U}{\pi} \frac{\sqrt{m} - \sqrt{n}}{\sqrt{n}} \right) - \half U^{-1} \left(U\frac{\sqrt{m} - \sqrt{n}}{\sqrt{n}}\right)^2 \widehat{g}'\left(\frac{U}{\pi} \frac{\sqrt{m} - \sqrt{n}}{\sqrt{n}} \right) + \dots.
\end{equation}
Note that each successive term has the same form as the leading term, but is smaller by a power of $U$, so we shall treat the generic term in what follows.

Thus it suffices to consider 
\begin{equation*}
 Z_3(m,n) = \e{(\sqrt{m}-\sqrt{n})^2}{ab} U w_3\left(U\frac{\sqrt{m} - \sqrt{n}}{\sqrt{n}} \right),
\end{equation*}
where $w_3$ is compactly supported and satisfies \eqref{eq:fbound} with $Y=1$.  Let $z = \frac{\sqrt{m} - \sqrt{n}}{\sqrt{ab}}$, and $Z = \frac{1}{U} \sqrt{\frac{n}{ab}}$, so that
\begin{equation*}
 Z_3(m,n) = U e(z^2) w_3\leg{z}{Z}.
\end{equation*}
Suppose $Z \gg T^{\varepsilon}$ (in our application this is always satisfied since $ab \ll (\Delta T)^{-1+ \varepsilon} N$, $\Delta = T^{1-\varepsilon}$, and $U \leq \Delta$).  
Then by Fourier inversion,
\begin{equation*}
 Z_3(m,n) = \intR Y(v) \e{v(\sqrt{m} - \sqrt{n})}{\sqrt{ab}} dv,
\quad
 Y(v) = U \intR e(z^2-vz) w_3\leg{z}{Z} dz. 
\end{equation*}
With inspiration here from pp. 431-432 of \cite{Sarnak},
we want to apply the Plancherel formula.  The function $e(z^2)$ is not $L^2$ so instead we argue directly by using $e(z^2) = \frac{e^{\pi i/4}}{\sqrt{2}} \intR e(tz - \frac{t^2}{4}) dt$ (which can be checked directly, the integral seen to converge uniformly on integration by parts) and reversing the order of integration.  Thus
\begin{equation*}
 Y(v) = \frac{e^{\pi i/4}}{\sqrt{2}} UZ  \intR e\Big(-\big(\frac{v+y}{2}\big)^2  \Big) \widehat{w_3}(-yZ) dy.
\end{equation*}
Expanding the square, we get
\begin{equation*}
 Y(v) = \sqrt{2} e^{\pi i/4} UZ e\Big(\frac{-v^2}{4}\Big) \intR e\left(-y^2 -vy\right) \widehat{w_3}(-2yZ) dy.
\end{equation*}
Next truncate the integral at $T^{-\varepsilon}$ (with negligible error), expand $e(-y^2)$ into a Taylor series 
taking $O(1/\varepsilon)$ terms so that the remainder is $O(T^{-2009})$, and then extend the integral back to $\mr$.  This gives an asymptotic expansion for $Y(v)$, with leading-order term say $Y_1$ given by
\begin{equation*}
 Y_1(v) = \sqrt{2} e^{\pi i/4} U Z e\Big(\frac{-v^2}{4}\Big) \intR e(-vy) \widehat{w_4}(-2yZ) dy = U \frac{e^{\pi i/4}}{\sqrt{2}} e\Big(\frac{-v^2}{4}\Big) w_3\leg{v}{2Z}.
\end{equation*}
The lower-order terms are similar but multiplied by powers of $Z^{-1}$, and with derivatives of $w_3$ replacing $w_3$, which of course are also compactly-supported and satisfy \eqref{eq:fbound} with $Y=1$.  Thus we conclude that
\begin{equation}
\label{eq:Z2}
 Z_2(m,n) \sim \frac{e^{\pi i/4}}{\sqrt{2}} \intR e\Big(\frac{-v^2}{4U^2}\Big) w_3\leg{v}{2UZ} \e{v(\sqrt{m} - \sqrt{n})}{U\sqrt{ab}} dv.
\end{equation}
Combining \eqref{eq:Z1} and \eqref{eq:Z2}, we get an asymptotic expansion for $J$ with leading-order term $J_1$ of the form
\begin{multline}
\label{eq:J1UB}
 |J_1| \ll \int_{-T^{\varepsilon}}^{T^{\varepsilon}} \int_{|v| \ll \sqrt{\frac{N}{ab}}} 
\Big|\sum_m b_m e\Big(\frac{\overline{x}m}{b} \Big) m^{-iy/2} e\Big(\frac{v \sqrt{m}}{U\sqrt{ab}} \Big) \Big|
\\
\Big|\sum_{n} b_n e\Big(\frac{  (\overline{x+a}) n}{b} \Big) n^{iy/2}  e\Big(\frac{v \sqrt{n}}{U\sqrt{ab}} \Big)  w_3\big(\frac{v}{ \sqrt{\frac{n}{ab}}} \big) \Big| dv dy.
\end{multline}
Using the simple inequality $|A| |B| \leq \half |A|^2 + \half |B|^2$, we obtain a bound for $|J_1|$ with two similar terms, one involving $w_3$, the other without $w_3$.  Consider the term with $w_3$, the other term being similar yet easier.  Inserting this bound into \eqref{eq:6.6}, we get a term of the form
\begin{equation*}
 \frac{\Delta T}{\sqrt{N}} \sum_{ab \leq \frac{N T^{\varepsilon}}{\Delta T}} \frac{1}{\sqrt{ab}} \int_{-T^{\varepsilon}}^{T^{\varepsilon}} 
\intR
\sum_{\substack{x \shortmod{b} \\ (x(x+a),b)=1}} \Big|\sum_{n} b_n e\Big(\frac{  (\overline{x+a}) n}{b} \Big) n^{iy/2}  e\Big(\frac{v \sqrt{n}}{U\sqrt{ab}} \Big) w_3\big(\frac{v}{ \sqrt{\frac{n}{ab}}} \big) \Big|^2 dv dy.
\end{equation*}
Now we simply drop the condition $(x,b) = 1$ by positivity, change variables $x \rightarrow x-a$ and change variables $ v \rightarrow \frac{\sqrt{N}}{\sqrt{ab}} v$.  In this way we arrive at a term of the form
\begin{equation*}
 \Delta T \sum_{ab \leq \frac{N T^{\varepsilon}}{\Delta T}} \frac{1}{ab} \int_{-T^{\varepsilon}}^{T^{\varepsilon}} \int_{|v| \ll 1} \sum_{\substack{x \shortmod{b} \\ (x,b)=1}} \Big|\sum_{n} b_n e\Big(\frac{  \overline{x} n}{b} \Big) n^{iy/2}  e\Big(\frac{v \sqrt{nN}}{Uab} \Big)  w_3\big(v \sqrt{N/n} \big) \Big|^2 dv dy,
\end{equation*}
which after summing over $l$ and bounding the $y$-integral by its length times the supremum over $y$ would give the desired form for Lemma \ref{lemma:K00}, except for the presence of $w_3(v \sqrt{\frac{N}{n}})$.  However, we can remove this dependence by a simple Mellin inversion similarly to how we handled $Z_1(m,n)$, so we omit the details (recall that the support on $b_n$ implicitly has $n \asymp N$).  The lower-order terms involving powers of $Z^{-1}$ can also be seen to have the same form.
\end{proof}

%


We shall use different methods of estimation depending on the sizes of $a$ and $b$.  Let $K_{A,B} = K_{A,B}(U, T, \Delta, a_{n,l})$ denote the same sum as $K_{00}$ given by \eqref{eq:K00def} but with $A < a \leq 2A$ and $B < b \leq 2B$, 
so that
\begin{equation}
 K_{0,0}(U, T, \Delta; a_{n,l}) \ll \log^2{T} \sup_{AB \leq \frac{N T^{\varepsilon}}{\Delta T}} K_{A,B}(U, T, \Delta, a_{n,l}),
\end{equation}
where \eqref{eq:cbound} translates to give the condition
\begin{equation}
 \label{eq:AB}
AB \ll \frac{N T^{\varepsilon}}{\Delta T}.
\end{equation}
By changing variables $v \rightarrow \frac{ab}{AB} v$ and summing trivially over $a$, note that
\begin{equation}
\label{eq:8.18}
 K_{A,B} \ll T^{\varepsilon} \frac{\Delta T}{B}  \int\limits_{v \ll 1}  \sum_{b \asymp B} \sum_{\substack{x \shortmod{b} \\ (x,b)=1}} \Big| \sum_m b_m e\Big(\frac{\overline{x}m}{b} \Big) m^{iy_0} e\Big(\frac{v \sqrt{mN}}{UAB} \Big) \Big|^2 dv.
\end{equation}
We can immediately apply Lemma \ref{lemma:variantsieve}, noting as well that $|b_m|^2 \ll |a_m|^2$, to get
\begin{mylemma}
\label{lemma:KABsmall}
We have
\begin{equation}
\label{eq:K1first}
 K_{A,B}(U, T, \Delta;a_{n}) \ll  T^{\varepsilon} \Delta T (B + UA) \sum_{n \leq N} |a_n|^2 \ll T^{\varepsilon} \Big(\frac{T^3}{A} + T^2 U A \Big) \sum_{n \leq N} |a_n|^2.
\end{equation}
\end{mylemma}
This estimate is sufficient for \eqref{eq:K00bound} for $A$ small ($A \ll U^{-1} T^{1 + \varepsilon}$), so we henceforth assume
\begin{equation}
 \label{eq:Alower}
A \geq T^{\varepsilon},
\end{equation}
which simplifies some later work.  Also, notice that we used no special properties of the coefficients $a_n$ so far.

\section{Voronoi summation}
In order to improve on Lemma \ref{lemma:KABsmall} we resort to use special properties of the coefficients $a_n$.  Our tool is the $GL(3)$ Voronoi summation formula proved by \cite{MS}.  We will state this important formula in a form developed by X. Li \cite{Li}.
\begin{mytheo}[Miller-Schmid]  Let $\psi$ be a smooth function with compact support on the positive reals.  Then
 \begin{multline}
\label{eq:Voronoi}
  \sum_{n} A(l,n) e\Big(\frac{n \overline{x}}{b} \Big) \psi(n) = \frac{b \pi^{-\frac52}}{4i} \sum_{n_1 | bl} \sum_{n_2 > 0} \frac{A(n_2,n_1)}{n_1 n_2} S\Big(lx, n_2 ; \frac{bl}{n_1} \Big) \Psi_1 \Big(\frac{n_2 n_1^2}{b^3 l}\Big)
\\
+ \frac{b \pi^{-\frac52}}{4i} \sum_{n_1 | bl} \sum_{n_2 > 0} \frac{A(n_2,n_1)}{n_1 n_2} S\Big(lx, -n_2 ; \frac{bl}{n_1} \Big) \Psi_2 \Big(\frac{n_2 n_1^2}{b^3 l}\Big),
 \end{multline}
for certain integral transforms $\Psi_1$ and $\Psi_2$.
\end{mytheo}
We need an explicit asymptotic expansion of $\Psi_1$ and $\Psi_2$, which is provided by Lemma 6.1 of \cite{Li2}  (generalizing Lemma 3 of Ivi\'{c} \cite{Ivic}).  Each of $\Psi_1$ and $\Psi_2$ is a linear combination of two other functions $\Psi_0(x)$ and $x^{-1} \Psi_{0,0}(x)$, say, where each has similar asymptotic behavior, so it suffices to treat $\Psi_0(x)$.
\begin{mylemma}[Ivi\'{c}, Li]  \label{lem:10.2} Suppose $\psi(r)$ is supported on $[N, 2N]$.  Then 
 there exist constants $c_{j, \pm}$ such that
\begin{equation}
 \Psi_0(x) = \sum_{j=1}^{L} \sum_{\pm} c_{j, \pm} x \int_0^{\infty} \psi(r) e(\pm 3 (xr)^{1/3}) \frac{dr}{(xr)^{j/3}} + O\left((xN)^{\frac{-L + 2}{3}} \right).
\end{equation}
\end{mylemma}
An easy contour shift argument shows that $\Psi_0(x)$ has rapid decay for $xN \rightarrow \infty$, and is bounded for $xN \ll 1$ (see the original expression (6.12) of \cite{Li2}).

Write \eqref{eq:8.18} in the form
\begin{equation*}
K_{A,B} \ll \frac{\Delta T^{1 + \varepsilon}}{B}  \int\limits_{v \ll 1}  \sum_{b \asymp B} \sum_{\substack{x \shortmod{b} \\ (x,b)=1}} \left|V(v; x, b)  \right|^2 dv,
\end{equation*}
where
\begin{equation}
V(v; x, b)= N^{-\half} \sum_m A(l,m) e\Big(\frac{xm}{b} \Big) m^{iy_0} e\Big(\frac{v \sqrt{mN}}{UAB} \Big) w(m),
\end{equation}
where $w$ is a function as in Lemma \ref{lemma:cleaning2}; notice that we wrote $n^{-\half} = N^{-\half} \leg{N}{n}^{\half}$, and absorbed the latter term into the weight function $w$.

Applying the Voronoi summation formula to $V(v;x,b)$ with $\psi(r) = r^{iy_0} e\big(\frac{v \sqrt{rN}}{U AB} \big) w(r)$, we obtain $V(v;x,b) = V_1(v;x,b) + V_2(v;x,b)$, say, corresponding to the two terms on the right hand side of \eqref{eq:Voronoi}.  We accordingly write $K_{A,B} \ll K_{A,B}^{+} + K_{A,B}^{-}$.  Changing variables by $x \rightarrow -x$ shows that $K_{A,B}^{-}$ is of a form similar to that of $K_{A,B}^{+}$, so we shall henceforth only treat $K_{A,B}^{+}$.  

First we claim we may assume $xN \gg T^{\varepsilon}$.  Otherwise, using $b \ll \frac{NT^{\varepsilon}}{AT^2}$, then
\begin{equation*}
 n_1^2 n_2 \ll \frac{b^3 l}{N} \ll T^{\varepsilon} \frac{N^2 l}{A^3 T^6}.
\end{equation*}
Recalling that $N \ll T^{3 + \varepsilon}/L^2$, we get that
$n_1^2 n_2 \ll \frac{T^{\varepsilon}}{L^3 A^3}$.
By \eqref{eq:Alower}, this condition is never satisfied.

Since $xN \gg T^{\varepsilon}$, Lemma \ref{lem:10.2} gives an asymptotic expansion of $V_1$.  As usual, 
we treat the leading-order term, say $K_{A,B}^{0}$, which takes the form
\begin{equation}
\label{eq:K00}
 K_{A,B}^{0} = \frac{\Delta T^{1 + \varepsilon}}{B}  \int\limits_{v \ll 1}  \sum_{b \asymp B} \sum_{\substack{x \shortmod{b} \\ (x,b)=1}} \Big|\sum_{n_1 | bl} \sum_{n_2 > 0} \frac{A(n_2, n_1)}{n_1 n_2} S\big(lx, n_2 ; \frac{bl}{n_1}\big) \Phi\Big(\frac{n_2 n_1^2}{b^3 l}\Big) \Big|^2 dv,
\end{equation}
where $\Phi$ is a function of the form
\begin{equation*}
 \Phi(\lambda) = \frac{b}{\sqrt{N}} \lambda \int_0^{\infty} w(r) r^{iy_0} e\Big( -3 (\lambda r)^{1/3} + \frac{u \sqrt{rN}}{UAB} \Big) \frac{dr}{(\lambda r)^{1/3}},
\end{equation*}
where $w$ is smooth, supported on $[N, 2N]$, satisfying \eqref{eq:fbound} with $Y=N$,
 and $y_0 \ll T^{\varepsilon}$, possibly after changing variables $v \rightarrow -v$ or $y_0 \rightarrow -y_0$.  The change of variables $r \rightarrow Nr$ gives, with $w_N(r) = w(Nr)$, 
\begin{equation}
\label{eq:10.9}
 \Phi(\lambda) = \frac{b}{\sqrt{N}} (\lambda N)^{2/3} N^{iy_0} \int_0^{\infty} w_N(r) r^{-\frac13 + iy_0} e\Big(-3 (\lambda rN)^{1/3} + \frac{v \sqrt{r}N}{UAB} \Big) dr.
\end{equation}

Our plan now is to express $K_{A,B}^{0}$ into a form where we can apply Lemma \ref{lemma:variantsieve}; the $v$-integral is critical for an extra saving effect, and as such it is important to understand the phase of the integral transform $\Phi$ (not just its magnitude).

\section{Asymptotic behavior of $\Phi(x)$}
Under the assumption $xN \gg T^{\varepsilon}$ (with $y_0 \ll T^{\varepsilon/100}$, say), the first term in the exponential in \eqref{eq:10.9} dominates over the phase of $y^{iv_0}$.  Unless the two terms in the exponential are of the same order of magnitude and of opposite signs (in particular, $v$ must be positive), then an easy integration by parts argument shows that $\Phi(\lambda)$ is negligible (smaller than $T^{-2009}$).  That is, $\Phi(\lambda)$ is small unless
\begin{equation}
\label{eq:xsize}
 \lambda \asymp \frac{v^3 N^2}{(UAB)^3},
\end{equation}
with certain absolute implied constants.  
Now suppose \eqref{eq:xsize} holds.
 
We shall treat general integrals of the form
\begin{equation*}
I = \int_0^{\infty} f(y) e(\alpha y^{1/2} - \beta y^{1/3}) dy,
\end{equation*}
where $\alpha, \beta > 0$, $\alpha \asymp \beta$, and $f$ satisfies
\begin{equation}
\label{eq:gbound}
\text{$f$ is smooth of compact support on $\mr^{+}$, satisfying } f^{(j)}(y) \ll T_0^{j}, 
\end{equation}
for some parameter $1 \leq T_0 \ll |\alpha|^{1/100}$.  The stationary phase method easily gives the main term for $I$, but a search of the literature did not find an adequate asymptotic expansion.  In this section we show that $I$ has an asymptotic expansion (as $\alpha \rightarrow \infty$) with leading term equal to
\begin{equation}
\label{eq:Iasymptotic}
 I \sim \frac{6 \leg{2\beta}{3\alpha}^5}{(2\beta)^{\half}} e\left(\frac{-4 \beta^3}{27 \alpha^2} +\frac18 \right)
f\Big(\big(\frac{2\beta}{3\alpha}\big)^6\Big) ,
\end{equation}
and where the lower-order terms have the same phase, but are smaller by powers of $\alpha$. 

Applying \eqref{eq:Iasymptotic} to $\Phi$, we find that
\begin{equation*}
 \Phi(\lambda) \sim b \sqrt{\lambda} h\leg{(UAB)^3 \lambda}{v^3 N^{2}} \e{-4\lambda (UAB)^2}{v^2 N} \leg{\lambda^2}{v^6}^{iy_0} z(N,U,A,B),
\end{equation*}
where $h$ is a smooth function of compact support on $\mr^{+}$, satisfying \eqref{eq:fbound} with $Y=1$, and $z(N, U, A, B)$ is some bounded function (not depending on either $v$ or $\lambda$).  Noting $b \sqrt{\lambda} = \sqrt{n_1 n_2} \frac{1}{\sqrt{\frac{bl}{n_1}}}$ and inserting this expression into \eqref{eq:K00}, we obtain
\begin{multline}
\label{eq:K002}
K_{A,B}^{0} \ll \frac{\Delta T^{1 + \varepsilon}}{B}  \int\limits_{v \ll 1}  \sum_{b \asymp B} \sum_{\substack{x \shortmod{b} \\ (x,b)=1}} 
\\
\Big|\sum_{n_1 | bl} \sum_{n_2} \frac{A(n_2, n_1)}{\sqrt{n_1 n_2}} \frac{S\big(lx, n_2 ; \frac{bl}{n_1}\big)}{\sqrt{bl/n_1}} h\leg{(UAB)^3 n_2 n_1^2}{v^3 b^3 l N^{2}} \e{-4 n_2 n_1^2 (UAB)^2}{v^2 b^3 lN} (n_1 n_2^2)^{2iy_0} \Big|^2 dv.
\end{multline}
We will continue with this expression in the following section.

\begin{proof}[Proof of \eqref{eq:Iasymptotic}]
Our goal is to use known properties of the Airy function, using ideas similar to those appearing in Section \ref{section:separation}.  First apply the change of variables $y = t^6$ to get
\begin{equation*}
 I = \int_0^{\infty} h(t) e(\alpha t^3 - \beta t^2) dt,
\end{equation*}
where $h(t) =  f(t^6) (6 t^5)$ satisfies \eqref{eq:gbound}.  We will show
\begin{equation}
\label{eq:Iasymptotic2}
 I \sim \frac{1}{(2\beta)^{\half}} e\left(\frac{-4 \beta^3}{27 \alpha^2} +\frac18 \right)
h\left(\frac{2\beta}{3\alpha}\right),
\end{equation}
which immediately implies \eqref{eq:Iasymptotic}.

Next change variables $t \rightarrow t + \beta/(3 \alpha)$ to get
\begin{equation*}
 I = e\big(\frac{-2 \beta^3}{27 \alpha^2} \big) \intR h\big(t + \frac{\beta}{3 \alpha}\big) e\big(\alpha t^3 -\frac{\beta^2}{3\alpha} t\big) dt.
\end{equation*}
Now use Fourier inversion on $h_{\alpha, \beta}(t) = h(t + \frac{\beta}{3 \alpha})$ (again, $h_{\alpha,\beta}$ satisfies \eqref{eq:gbound} except its support may include negative reals) and reverse the orders of integration (justified by uniform convergence following from integration by parts) to get 
\begin{equation*}
 I = e\big(\frac{-2 \beta^3}{27 \alpha^2} \big) \intR \widehat{h_{\alpha, \beta}}(-t)  \intR e\big(\alpha y^3 -(\frac{\beta^2}{3\alpha} + t) y\big) dy dt.
\end{equation*}
From a change of variables and some simple symmetry arguments, note that
\begin{equation*}
 \intR e\big(\alpha y^3 -(\frac{\beta^2}{3\alpha} + t) y\big) dy = \frac{2}{(6 \pi \alpha)^{1/3}} \int_0^{\infty} \cos\big(\frac13 y^3 - 2 \pi \frac{\frac{\beta^2}{3\alpha} + t}{(6 \pi \alpha)^{1/3}} y\big) dy,
\end{equation*}
which using the definition of the Airy function $\Ai(x)$ can be expressed as
\begin{equation*}
 \frac{2\pi}{ (6 \pi \alpha)^{1/3}} \Ai\big(-2\pi \frac{\frac{\beta^2}{3\alpha} + t}{(6 \pi \alpha)^{1/3}}\big).
\end{equation*}
In terms of $I$, we thus have
\begin{equation*}
I = \frac{2\pi}{(6 \pi \alpha)^{1/3}} e\big(\frac{-2 \beta^3}{27 \alpha^2} \big) \intR \widehat{h_{\alpha, \beta}}(-t)  \Ai\big(-2 \pi \frac{\frac{\beta^2}{3\alpha} + t}{(6 \pi \alpha)^{1/3}}\big) dt.
\end{equation*}
Next change variables by $t \rightarrow \frac{\beta^2}{3 \alpha} t$ to get
\begin{equation*}
I = \frac{2\pi}{(6 \pi \alpha)^{1/3}} e\big(\frac{-2 \beta^3}{27 \alpha^2} \big) 
\intR 
\frac{\beta^2}{3 \alpha} \widehat{h_{\alpha, \beta}}\big(-\frac{\beta^2}{3 \alpha}t\big)  
\Ai\big(-2 \pi \frac{\frac{\beta^2}{3\alpha}(1 + t)}{(6 \pi \alpha)^{1/3}}\big) dt.
\end{equation*}
We insert the asymptotic expansion for the Airy function at large negative argument (see (4.07) of \cite{Olver}), namely
\begin{equation*}
 \Ai(-x) \sim \frac{1}{\sqrt{\pi} x^{\frac14}} \Big[\cos\big( \frac{2}{3} x^{3/2} -\frac{\pi}{4} \big) \sum_{k} \frac{c_{2k}}{x^{3k}} +  \sin\big( \frac{2}{3} x^{3/2} -\frac{\pi}{4} \big) \sum_{k} \frac{c_{2k+1}}{x^{\frac32 (2k+1)}}\Big],
\end{equation*}
for certain explicit constants $c_k$ (in particular, $c_0 = 1$).  
To justify this, we note that integration by parts shows $\widehat{h_{\alpha, \beta}}(y)  \ll (y^{-1} T_0)^j$ where the implied constant depends on the support of $h$. Thus we may truncate the integral at $t \ll \alpha^{-2/3}$, say, with a negligible error of size any power of $\alpha^{-1}$.
We then have
\begin{equation*}
 I \sim \frac{2\sqrt{\pi}}{(6 \pi \alpha)^{1/3}} \e{-2 \beta^3}{27 \alpha^2} 
\int_{-\alpha^{-2/3}}^{\alpha^{-2/3}}
\frac{\beta^2}{3 \alpha} \widehat{h_{\alpha, \beta}}\big(-\frac{\beta^2}{3 \alpha}t\big)  
\frac{\cos\big( \frac{2}{3} \big(\frac{2 \pi \frac{\beta^2}{3\alpha}(1 + t)}{(6 \pi \alpha)^{1/3}}\big)^{3/2} -\frac{\pi}{4} \big)}{\big(\frac{2 \pi \frac{\beta^2}{3\alpha}(1 + t)}{(6 \pi \alpha)^{1/3}} \big)^{1/4}}
dt,
\end{equation*}
where the lower-order terms have a similar shape but are multiplied by powers of $\frac{\alpha^2}{\beta^3} (1+t)^{-3/2} \asymp \alpha^{-1} (1 + t)^{-3/2}$.  The terms with $\cos$ replaced by $\sin$ are treated similarly, so we work with $\cos$ only.  This expression simplifies as
\begin{equation*}
 I \sim \frac{2^{\half}}{\beta^{\frac12}} e\big(\frac{-2 \beta^3}{27 \alpha^2} \big)
\intR 
\frac{\beta^2}{3 \alpha} \widehat{h_{\alpha, \beta}}\big(-\frac{\beta^2}{3 \alpha}t\big) (1+t)^{-\frac14}
\cos\big( \frac{2}{3} \big(\frac{2 \pi \frac{\beta^2}{3\alpha}(1 + t)}{(6 \pi \alpha)^{1/3}} \big)^{3/2} -\frac{\pi}{4} \big)
dt.
\end{equation*}
We expand $(1+t)^{-\frac14}$ into a Taylor series, developing the expansion futher.

Now write $\cos{x} = \frac{1}{2 } ( e^{ix} + e^{-ix})$, and write $I \sim I_{+} + I_{-}$ correspondingly.
Thus
\begin{equation*}
I_{\pm} \sim \frac{1}{(2\beta)^{\half}} e\big(\frac{-2 \beta^3}{27 \alpha^2} \big)
\int_{-\alpha^{-2/3}}^{\alpha^{-2/3}}
\frac{\beta^2}{3 \alpha} \widehat{h_{\alpha, \beta}}\big(-\frac{\beta^2}{3 \alpha}t\big)
e\Big( \pm  \big(\frac{2 \beta^3}{27 \alpha^2}(1 + t)^{\frac32} - \frac{1}{8} \big) \Big)
dt.
\end{equation*}
Next we take a Taylor series for $(1 + t)^{3/2} = 1 + \frac{3}{2} t + \dots$ in the exponential.  The quadratic and higher terms are small (much less than $1$) so we take a Taylor series expansion for the exponential of these terms, giving another asymptotic expansion with leading-order term
\begin{equation*}
I_{\pm} \sim \frac{\e{\mp 1}{8}}{(2\beta)^{\half}} e\big(\frac{-2 \beta^3 (1 \mp 1)}{27 \alpha^2} \big)
\int_{-\alpha^{-2/3}}^{\alpha^{-2/3}}
\frac{\beta^2}{3 \alpha} \widehat{h_{\alpha, \beta}}\big(-\frac{\beta^2 }{3 \alpha}t\big)
e\Big( \pm  \big(\frac{ \beta^3}{9 \alpha^2} t  \big) \Big)
dt.
\end{equation*}
Extending the integral back to $\mr$ and changing variables back via $t \rightarrow \frac{3 \alpha}{\beta^2} t$ gives
\begin{equation*}
I_{\pm} \sim \frac{\e{\mp 1}{8}}{(2\beta)^{\half}} e\big(\frac{-2 \beta^3 (1 \mp 1)}{27 \alpha^2} \big)
\intR 
\widehat{h_{\alpha, \beta}}\left(-t\right)
e\big( \pm  (\frac{ \beta}{ 3\alpha} t  ) \big)
dt.
\end{equation*}
Calculating the integral in terms of $h$, we get
\begin{equation*}
I_{\pm} \sim \frac{\e{\mp 1}{8}}{(2\beta)^{\half}} e\big(\frac{-2 \beta^3 (1 \mp 1)}{27 \alpha^2} \big)
h_{\alpha, \beta}\big(\mp \frac{\beta}{3\alpha} \big).
\end{equation*}
Note $h_{\alpha,\beta}\left( \mp \frac{\beta}{3\alpha}\right) = h\left(\frac{\beta}{3\alpha} (1 \mp 1) \right)$. 
Since $h$ has support on the positive reals, we have $h(0) = 0$, in which case only $I_{-}$ contributes to $I$, so then
\begin{equation*}
 I \sim \frac{1}{(2\beta)^{\half}} e\big(\frac{-4 \beta^3}{27 \alpha^2} +\frac18 \big)
h\big(\frac{2\beta}{3\alpha}\big).
\end{equation*}
One easily checks that this is the expected main term one obtains from stationary phase.
\end{proof}

\section{Cleaning}
We have reduced the problem of estimating $K_{A,B}$ (originally given by \eqref{eq:8.18}) to estimating the rather messy expression $K_{A,B}^{0}$ given by \eqref{eq:K002}, in the sense that any bound for $K_{A,B}^{0}$ with a general compactly-supported weight function $h$ satisfying \eqref{eq:fbound} with $Y=1$, is also a bound on $K_{A,B}$ (plus a negligible error term of size, say $O(T^{-100})$ which shall be dwarfed by our upper bound on $K_{A,B}$).

We shall make some preliminary transformations to clean up this expression for $K_{A,B}^{0}$.  The reader interested in the essential details should consider the crucial case $v \asymp 1$, $l=n_1 = 1$, which greatly simplifies the forthcoming calculations.
\begin{mylemma}
\label{lemma:Lreduction}
 Suppose that $K_{A,B}^{0}$ is any expression of the form \eqref{eq:K002} with $h$ satisfying \eqref{eq:fbound} with $Y=1$, and let $\mathcal{L} = \mathcal{L}(A, B, L, T, U, \Delta, N_1, N_2)$ be an expression of the form
\begin{equation}
\label{eq:Ldef}
 \mathcal{L} = \frac{\Delta T}{B} \frac{W'}{L} \sum_{l \asymp L} \sum_{b \asymp B} \sum_{\substack{x \shortmod{b} \\ (x,b) = 1}} \sum_{\substack{n_1 | bl \\ n_1 \asymp N_1}} \int_{v \asymp 1} 
\Big| \sum_{n_2 \asymp N_2} \frac{b(n_1, n_2)}{\sqrt{n_1 n_2}} \frac{S(lx, n_2 ;\frac{bl}{n_1})}{\sqrt{bl/n_1}} \e{v n_2}{W} \Big|^2 dv,
\end{equation}
where $b(n_1, n_2)$ are complex numbers satisfying $|b(n_1, n_2)| \ll |A(n_1, n_2)|$, and where
\begin{equation}
\label{eq:WW'}
 W = \frac{N_2^{2/3} B L^{1/3} }{N_1^{2/3} N^{1/3}}, 
\qquad W' = \frac{UA N_2^{1/3} N_1^{2/3}}{ L^{1/3} N^{2/3}}.
\end{equation}
Then for some such choice of $b(n_1, n_2)$, we have
\begin{equation}
\sum_{l \asymp L} l^{-1} K_{A,B}^{0} \ll T^{\varepsilon} \sup_{A,B, N_1, N_2} \mathcal{L},
\end{equation}
where recall \eqref{eq:AB}, and where
 \begin{equation}
\label{eq:N1N2size}
 N_1^2 N_2 \ll \frac{L N^2}{(UA)^3}.
\end{equation}

\end{mylemma}

\begin{proof}
We begin with the representation \eqref{eq:K002} and apply Cauchy's inequality to take the sum over $n_1 | bl$ outside the absolute values, and divide up the sum over $n_2$ into dyadic intervals $n_2 \asymp N_2$, where in view of \eqref{eq:xsize},
\begin{equation}
 \label{eq:n2size}
N_2 \ll \frac{l N^2}{n_1^2 (UA)^3}.
\end{equation}
With the shorthand notation
\begin{equation}
 K' = \int\limits_{v \ll 1} \Big|\sum_{n_2 \asymp N_2} \frac{a(n_1, n_2)}{\sqrt{n_1 n_2}} \frac{S(lx, n_2 ; \frac{bl}{n_1})}{\sqrt{bl/n_1}} h\leg{(UAB)^3 n_2 n_1^2}{v^3 b^3 l N^{2}} \e{4 n_2 n_1^2 (UAB)^2}{v^2 b^3 lN} \Big|^2 dv,
\end{equation}
where $a(n_1, n_2)$ are certain complex numbers with the same absolute value as $A(n_2, n_1)$, we have
\begin{equation}
K_{A,B}^{0} \ll T^{\varepsilon} \frac{\Delta T}{B}    \sum_{b \asymp B} \sum_{\substack{x \shortmod{b} \\ (x,b)=1}}\sum_{n_1 | bl} K'.
\end{equation}

Next, locate $n_1 \asymp N_1$, so that \eqref{eq:N1N2size} holds, and recall $l \asymp L$.
For such $l$ and $n_1$, change variables in $v$ via
\begin{equation*}
 v \rightarrow \frac{UAB^{3/2} N_2^{1/3} n_1  L^{1/6}}{b^{3/2} l^{1/2} N^{2/3} N_1^{1/3}} v^{-1/2}, 
\end{equation*}
to get
\begin{equation*}
K' \ll  W'
\int\limits_{v \asymp 1} \Big|\sum_{n_2 \asymp N_2} \frac{a(n_1, n_2)}{\sqrt{n_1 n_2}} \frac{S(lx, n_2 ; \frac{bl}{n_1})}{\sqrt{bl/n_1}} h\Big( \big(\frac{v^3 b^3 l n_2^2 N_1^2}{B^3 L N_2^2 n_1^2} \big)^{1/2} \Big) \e{vn_2  }{W}  \Big|^2 dv,
\end{equation*}
where $W$ and $W'$ are given by \eqref{eq:WW'}, the restriction to $v \asymp 1$ is redundant to the support of $h$, and where we use positivity and the location of the variables to write $W'$ in terms of capital letters rather than lowercase letters.

Next separate the variables $v$, $b$, $l$, $n_1$, and $n_2$ in $h\Big( \big(\frac{v^3 b^3 l n_2^2 N_1^2}{B^3 L N_2^2 n_1^2} \big)^{1/2} \Big)$ by using the Mellin inversion formula; since $v$ is already located to be $\asymp 1$, and $b \asymp B$, $l \asymp L$, and $n_1 \asymp N_1$, $n_2 \asymp N_2$.  This can be done with a Mellin integral of length $O(T^{\varepsilon})$ which can be taken to the outside of $\sum_l l^{-l} K_{A,B}^{0}$ (after an application of Cauchy-Schwarz) and then bounded by its length times the supremum.  Effectively this simply changes the coefficients to say $b(n_1, n_2)$ having the same absolute values as $a(n_1, n_2)$.
\end{proof}

\section{Final step: applying the large sieve}
\begin{mylemma}
\label{lemma:Lbound}
 Suppose that $\mathcal{L}$ is given by \eqref{eq:Ldef} as in Lemma \ref{lemma:Lreduction}.  Then
\begin{equation}
 \mathcal{L} \ll T^{3 + \varepsilon} \sum_{ln \ll T^{3 + \varepsilon}} \frac{|b(l, n)|^2}{ln}.
\end{equation}
\end{mylemma}
Lemma \ref{lemma:Lbound} shall complete the proof of Theorem \ref{thm:mainthm''} by the reductions of Lemmas \ref{lemma:K0}, \ref{lemma:K00}, and \ref{lemma:Lreduction}, combined with the bound of Lemma \ref{lemma:KABsmall} for $A \leq T^{\varepsilon}$.

\begin{proof}
Let $d = (b, n_1)$ and change variables $b \rightarrow d b$, $n_1 \rightarrow d n_1$ to get
\begin{equation}
\label{eq:13.1}
\mathcal{L} \ll T^{\varepsilon} \frac{\Delta T}{B} \frac{W'}{L} \sum_{l \asymp L}  \int\limits_{v \asymp 1}  \sum_{\substack{d \ll N_1}} \sum_{b \asymp B/d} \sum_{\substack{x \shortmod{bd} \\ (x,bd)=1}} \sum_{\substack{n_1 | l \\ (n_1, b) = 1 \\ dn_1 \asymp N_1}} \Big|  \sum_{n_2 } \frac{b(dn_1,n_2)}{\sqrt{d n_1 n_2}} \frac{S(lx, n_2 ; \frac{bl}{n_1} )}{\sqrt{bl/n_1}} \e{vn_2}{W} \Big|^2 du.
\end{equation}
Write $\frac{l}{n_1}  = rs$ where $r |b^{\infty}$ (meaning all primes dividing $r$ also divide $b$) and $(s,b) = 1$.  Next change variables to eliminate $l$ and note that the sum over $x$ only depends modulo $b$ to give
\begin{equation}
\label{eq:13.2}
\mathcal{L} \ll T^{\varepsilon} \frac{\Delta T}{B} \frac{W'}{ L} \int\limits_{v \asymp 1}  \sum_{d \ll N_1} d \sum_{b \asymp \frac{B}{d}} \sum_{\substack{n_1 r s \asymp L \\(n_1 s, b)  = 1 \\ dn_1 \asymp  N_1}} \sum_{r |b^{\infty}} \frac{1}{brs} \mathcal{L}_1,
\end{equation}
where as shorthand
\begin{equation}
 \mathcal{L}_1 = \sum_{\substack{x \shortmod{b} \\ (x,b)=1}}  \big|  \sum_{n_2 } \frac{b(dn_1,n_2)}{\sqrt{dn_1 n_2}} S\left( n_1 rs x, n_2 ; brs \right) \e{vn_2}{W} \big|^2.
\end{equation}
Next we simplify $\mathcal{L}_1$, by showing
\begin{equation}
\label{eq:K1}
 \mathcal{L}_1 \leq  br^2s \sumstar_{h \shortmod{bs}} \big| \sum_{n_2} \frac{b(dn_1, rn_2)}{\sqrt{dn_1 r n_2}}  e\big(\frac{h n_2}{bs} \big) e\big(\frac{v r n_2}{W} \big) \big|^2.
\end{equation}

\begin{proof}[Proof of \eqref{eq:K1}]
From the multiplicativity relation for Kloosterman sums, we obtain
\begin{equation*}
S\left( n_1 rs x, n_2 ; brs \right) = S(n_1 r s \overline{s} x, n_2 \overline{s}, br) S(n_1 rs \overline{br} x, n_2 \overline{br} ;s) = S(n_1 r  x, n_2 \overline{s} , br) S(0, n_2 ;s) ,
\end{equation*}
which becomes $S(r x, n_2  ;br) S(0, n_2;s)$ after the change of variables $x \rightarrow s \overline{n_1} x$ (recall $n_1$ is coprime to $br$). Thus
\begin{equation*}
 \mathcal{L}_1 = \sum_{\substack{x \shortmod{b} \\ (x,b)=1}}  \big|  \sum_{m } b_m S\left( r x, m ; br \right) \big|^2, \text{ with }  b_m = \frac{b(dn_1, m)}{\sqrt{dn_1 m}} S(0, m;s) \e{vm}{W}.
\end{equation*}

Next we compute for arbitrary complex numbers $c_m$,
\begin{equation}
\label{eq:13.4}
\sum_{x \shortmod{b}} \big| \sum_m c_m S(rx, m ;br) \big|^2 = b \sum_{m_1, m_2} c_{m_1} \overline{c_{m_2}} \sumstar_{\substack{h_1, h_2 \shortmod{br} \\h_1 \equiv h_2 \shortmod{b}}} \e{h_1 m_1 - h_2 m_2}{br}.
\end{equation}
Change variables via $h_i = y + b z_i$, $i=1,2$, where $y$ runs modulo $b$ and $z_i$ runs modulo $r$.  Since $r |b^{\infty} $, the condition that $(h_i, br) = 1$ is equivalent to $(y, b) = 1$.  The sum over $z_i$ vanishes unless $r | m_i$, in which case the sum is $r$.  Thus \eqref{eq:13.4} equals
\begin{equation*}
b r^2 \sum_{r|m_1, m_2} c_{m_1} \overline{c_{m_2}} \sumstar_{y \shortmod{b}} \e{y (\frac{m_1}{r} - \frac{m_2}{r})}{b} = b r^2 \sumstar_{y \shortmod{b}} \Big| \sum_{r|m} c_{m} \e{y \frac{m}{r}}{b}\Big|^2,
\end{equation*}
and hence
\begin{equation*}
 \mathcal{L}_1 \leq br^2 \sumstar_{y \shortmod{b}} \big| \sum_{r | m} c_m \e{y \frac{m}{r}}{b} S(0,m;s) \big|^2, \text{ with }  c_m = \frac{a(dn_1, m)}{\sqrt{dn_1 m}} \e{vm}{W}.
\end{equation*}
By Cauchy's inequality, we have for any complex coefficients $b_l$ that
\begin{equation}
\label{eq:11.7}
\big| \sum_{l} b_l S(0,l;s) \big|^2 \leq s \sumstar_{h \shortmod{s}} \big|\sum_l b_l e\big(\frac{hl}{s} \big) \big|^2.
\end{equation}
Note that $S(0, l ;s) = S(0,\frac{l}{r} ;s)$ since $(r,s) = 1$.  Hence
\begin{equation*}
 \mathcal{L}_1 \leq br^2s \sumstar_{h \shortmod{s}} \sumstar_{y \shortmod{b}} \Big| \sum_{r | m} c_m \e{y \frac{m}{r}}{b} \e{h \frac{m}{r}}{s} \Big|^2, 
\end{equation*}
which gives \eqref{eq:K1} using the Chinese remainder theorem and changing variables $m \rightarrow mr$.
\end{proof}

Picking back up the chain of reasoning, we insert \eqref{eq:K1} into \eqref{eq:13.2}, getting
\begin{equation}
\label{eq:13.6}
\mathcal{L} \ll  T^{\varepsilon} \frac{\Delta T}{B} \frac{W'}{L} \int\limits_{v \asymp 1} \sum_{d \ll N_1} d  \sum_{b \asymp \frac{B}{d}} \sum_{\substack{n_1 r s \asymp L \\(n_1 s, b)  = 1 \\ dn_1 \asymp  N_1}} \sum_{r |b^{\infty}} r \sumstar_{y \shortmod{bs}}  
\Big|  \sum_{n_2} \frac{b(dn_1,rn_2)}{\sqrt{dn_1 r n_2}} e\big(\frac{y n_2}{bs}\big) e\big(\frac{vr n_2}{W} \big) \Big|^2 dv.
\end{equation}
Next relax the condition that $r |b^{\infty}$ and for convenience locate the variable $s \asymp S$, where 
\begin{equation}
 N_1 rS \asymp  d L.
\end{equation}
We need to bound the supremum over $S$, where note that we may assume $S \ll L$.
In addition change variables $bs \rightarrow c$, getting
\begin{equation}
\label{eq:13.9}
\mathcal{L} \ll T^{\varepsilon} \frac{\Delta T}{B} \frac{W'}{ L} 
 \sum_{\substack{dn_1 \asymp N_1 }} d \sum_{r \ll \frac{L}{n_1 S}}  r \int\limits_{v \asymp 1} \sum_{c \asymp \frac{BS}{d}} \sumstar_{h \shortmod{c}}  \Big|  \sum_{r n_2 \asymp N_2} \frac{b(dn_1,r n_2)}{\sqrt{dn_1 r n_2}} \e{y n_2}{c} \e{v n_2 }{W/r} \Big|^2 dv,
\end{equation}
for some value of $S$.  
Finally, an application of Lemma \ref{lemma:variantsieve} gives
\begin{equation*}
\mathcal{L} \ll T^{\varepsilon} \frac{\Delta T}{B} \frac{W'}{ L} 
\sum_{dn_1 \asymp  N_1} d \sum_{r \ll \frac{L}{n_1 S}}  r \Big( \big(\frac{BS}{d} \big)^2 + \frac{W}{r}\Big)   \sum_{\substack{r n_2 \asymp N_2}}  \frac{|b(dn_1, rn_2)|^2}{ dn_1 r n_2}.
\end{equation*}
We simplify this by first letting $rn_2 \rightarrow n_2$ be a new variable, getting
\begin{equation*}
 \mathcal{L} \ll T^{\varepsilon} \frac{\Delta T}{B} \frac{W'}{ L} \sum_{dn_1 \asymp  N_1} d \Big(\frac{L}{n_1 S} \big(\frac{BS}{d} \big)^2 + W \Big) \sum_{n \asymp N_2}
\frac{|b(dn_1, n_2)|^2}{ dn_1 n_2}.
\end{equation*}
Similarly, let $dn_1 \rightarrow n_1$ be a new variable to get
\begin{equation}
\label{eq:11.15}
\mathcal{L} \ll  T^{\varepsilon} \frac{\Delta T}{B} \frac{W'}{ L}  \Big(\frac{L B^2 S}{N_1} +  W N_1 \Big)  \sum_{l^2 n \ll N_1^2 N_2} \frac{|b(n_1, n_2)|^2}{n_1 n_2}.
\end{equation}
Note that from \eqref{eq:N1N2size} we have
 $W' \ll 1$
and
\begin{equation*}
W' W = \frac{UAB N_2}{N} \ll \frac{BLN}{(UA)^2 N_1^2}.  
\end{equation*}
Thus using $S \ll L$, the expression \eqref{eq:11.15} simplifies as
\begin{equation}
\label{eq:11.17}
\mathcal{L} \ll T^{\varepsilon} \Delta T \Big( \frac{B L}{N_1} +  \frac{ N}{N_1 (UA)^2}\Big) \sum_{l^2 n \ll N_1^2 N_2} \frac{|b(l, n)|^2}{l n}.
\end{equation}
Recalling that $B \ll \frac{N T^{\varepsilon}}{A \Delta T}$, $L^2 N \ll T^{3 +\varepsilon}$, and $\Delta = U = T^{1-\varepsilon}$, finishes the proof.
\end{proof}

\end{document}